\documentclass[12pt]{amsart}
\usepackage{amscd,amssymb,amsmath,multicol,float,scalefnt,cancel,array,stmaryrd,graphicx,tikz,tikz-cd}
\usepackage[right=1.5in]{geometry}
\restylefloat{table}
\newcommand{\Mdef}[2]{\newcommand{#1}{\relax \ifmmode #2 \else $#2$\fi}}

\newcommand{\sm }{\wedge}

\newcommand{\tensor}{\otimes}

\newcommand{\map}{\mathrm{map}}
\newcommand{\Hom}{\mathrm{Hom}}

\newcommand{\Ext}{\mathrm{Ext}}
\Mdef{\bhom}{\mathbf{\hat{H}om}}

\Mdef{\Mod}{\mathrm{mod}}

\newtheorem{thm}{Theorem}[section]
\newtheorem{lemma}[thm]{Lemma}

\newtheorem{cor}[thm]{Corollary}

\theoremstyle{definition}

\newtheorem{defn}[thm]{Definition}

\newtheorem{notation}[thm]{Notation}

\newtheorem{example}[thm]{Example}

\newtheorem{remark}[thm]{Remark}

\newcommand{\qqed}{\qed \\[1ex]}
\renewenvironment{proof}[1][\hspace*{-.8ex}]{\noindent {\bf Proof #1:\;}}{\qqed}

\Mdef{\PH} {\Phi^H}
\Mdef{\PK} {\Phi^K}
\Mdef{\PL} {\Phi^L}
\Mdef{\PT} {\Phi^{\T}}

\Mdef{\ef}{E{\cF}_+}
\Mdef{\etf}{\widetilde{E}{\cF}}
\Mdef{\eg}{E{G}_+}
\Mdef{\etg}{\tilde{E}{G}}

\Mdef{\infl}{\mathrm{inf}}
\Mdef{\defl}{\mathrm{def}}
\Mdef{\res}{\mathrm{res}}
\Mdef{\ind}{\mathrm{ind}}
\Mdef{\coind}{\mathrm{coind}}

\Mdef{\univ}{\mathcal{U}}

\Mdef{\Fp}{\mathbb{F}_p}
\Mdef{\Zpinfty}{\Z /p^{\infty}}
\Mdef{\Zpadic}{\Z_p^{\wedge}}

\newcommand{\dichotomy}[2]{\left\{ \begin{array}{ll}#1\\#2 \end{array}\right.}

\newcommand{\lra}{\longrightarrow}

\newcommand{\lr}[1]{\langle #1 \rangle}

\Mdef{\we}{\mathbf{we}}
\Mdef{\fib}{\mathbf{fib}}
\Mdef{\cof}{\mathbf{cof}}
\Mdef{\BI}{\mathcal{BI}}

\Mdef{\B}{\mathbb{B}}
\Mdef{\C}{\mathbb{C}}
\Mdef{\D}{\mathbb{D}}
\Mdef{\E}{\mathbb{E}}
\Mdef{\T}{\mathbb{T}}
\Mdef{\F}{\mathbb{F}}
\Mdef{\G}{\mathbb{G}}
\Mdef{\I}{\mathbb{I}}
\Mdef{\N}{\mathbb{N}}
\Mdef{\Q}{\mathbb{Q}}
\Mdef{\R}{\mathbb{R}}
\Mdef{\bbS}{\mathbb{S}}
\newcommand{\Z}{\mathbb{Z}}

\Mdef{\bA}{\mathbb{A}}
\Mdef{\bB}{\mathbb{B}}
\Mdef{\bC}{\mathbb{C}}
\Mdef{\bD}{\mathbb{D}}
\Mdef{\bE}{\mathbb{E}}
\Mdef{\bF}{\mathbb{F}}
\Mdef{\bG}{\mathbb{G}}
\Mdef{\bH}{\mathbb{H}}
\Mdef{\bI}{\mathbb{I}}
\Mdef{\bJ}{\mathbb{J}}
\Mdef{\bK}{\mathbb{K}}
\Mdef{\bL}{\mathbb{L}}
\Mdef{\bM}{\mathbb{M}}
\Mdef{\bN}{\mathbb{N}}
\Mdef{\bO}{\mathbb{O}}
\Mdef{\bP}{\mathbb{P}}
\Mdef{\bQ}{\mathbb{Q}}
\Mdef{\bR}{\mathbb{R}}
\Mdef{\bS}{\mathbb{S}}
\Mdef{\bT}{\mathbb{T}}
\Mdef{\bU}{\mathbb{U}}
\Mdef{\bV}{\mathbb{V}}
\Mdef{\bW}{\mathbb{W}}
\Mdef{\bX}{\mathbb{X}}
\Mdef{\bY}{\mathbb{Y}}
\Mdef{\bZ}{\mathbb{Z}}

\Mdef{\cA}{\mathcal{A}}
\Mdef{\cB}{\mathcal{B}}
\Mdef{\cC}{\mathcal{C}}
\Mdef{\mcD}{\mathcal{D}} 
\Mdef{\cE}{\mathcal{E}}
\Mdef{\cF}{\mathcal{F}}
\Mdef{\cG}{\mathcal{G}}
\Mdef{\mcH}{\mathcal{H}} 
\Mdef{\cI}{\mathcal{I}}
\Mdef{\cJ}{\mathcal{J}}
\Mdef{\cK}{\mathcal{K}}
\Mdef{\mcL}{\mathcal{L}}

\Mdef{\cM}{\mathcal{M}}
\Mdef{\cN}{\mathcal{N}}
\Mdef{\cO}{\mathcal{O}}
\Mdef{\cP}{\mathcal{P}}
\Mdef{\cQ}{\mathcal{Q}}
\Mdef{\mcR}{\mathcal{R}}
\Mdef{\cS}{\mathcal{S}}
\Mdef{\cT}{\mathcal{T}}
\Mdef{\cU}{\mathcal{U}}
\Mdef{\cV}{\mathcal{V}}
\Mdef{\cW}{\mathcal{W}}
\Mdef{\cX}{\mathcal{X}}
\Mdef{\cY}{\mathcal{Y}}
\Mdef{\cZ}{\mathcal{Z}}

\Mdef{\ca}{\mathcal{a}}
\Mdef{\ct}{\mathcal{t}}

\Mdef{\At}{\tilde{A}}
\Mdef{\Bt}{\tilde{B}}
\Mdef{\Ct}{\tilde{C}}
\Mdef{\Et}{\tilde{E}}
\Mdef{\Ht}{\tilde{H}}
\Mdef{\Kt}{\tilde{K}}
\Mdef{\Lt}{\tilde{L}}
\Mdef{\Mt}{\tilde{M}}
\Mdef{\Nt}{\tilde{N}}
\Mdef{\Pt}{\tilde{P}}

\Mdef{\tA}{\tilde{A}}
\Mdef{\tB}{\tilde{B}}
\Mdef{\tC}{\tilde{C}}
\Mdef{\tE}{\tilde{E}}
\Mdef{\tH}{\tilde{H}}
\Mdef{\tK}{\tilde{K}}
\Mdef{\tL}{\tilde{L}}
\Mdef{\tM}{\tilde{M}}
\Mdef{\tN}{\tilde{N}}
\Mdef{\tP}{\tilde{P}}

\Mdef{\ft}{\tilde{f}}
\Mdef{\xt}{\tilde{x}}
\Mdef{\yt}{\tilde{y}}

\Mdef{\Ab}{\overline{A}}
\Mdef{\Bb}{\overline{B}}
\Mdef{\Cb}{\overline{C}}
\Mdef{\Db}{\overline{D}}
\Mdef{\Eb}{\overline{E}}
\Mdef{\Fb}{\overline{F}}
\Mdef{\Gb}{\overline{G}}
\Mdef{\Hb}{\overline{H}}
\Mdef{\Ib}{\overline{I}}
\Mdef{\Jb}{\overline{J}}
\Mdef{\Kb}{\overline{K}}
\Mdef{\Lb}{\overline{L}}
\Mdef{\Mb}{\overline{M}}
\Mdef{\Nb}{\overline{N}}
\Mdef{\Ob}{\overline{O}}
\Mdef{\Pb}{\overline{P}}
\Mdef{\Qb}{\overline{Q}}
\Mdef{\Rb}{\overline{R}}
\Mdef{\Sb}{\overline{S}}
\Mdef{\Tb}{\overline{T}}
\Mdef{\Ub}{\overline{U}}
\Mdef{\Vb}{\overline{V}}
\Mdef{\Wb}{\overline{W}}
\Mdef{\Xb}{\overline{X}}
\Mdef{\Yb}{\overline{Y}}
\Mdef{\Zb}{\overline{Z}}

\Mdef{\db}{\overline{d}}
\Mdef{\hb}{\overline{h}}
\Mdef{\qb}{\overline{q}}
\Mdef{\rb}{\overline{r}}
\Mdef{\tb}{\overline{t}}
\Mdef{\ub}{\overline{u}}
\Mdef{\vb}{\overline{v}}

\Mdef{\hc}{\hat{c}}
\Mdef{\he}{\hat{e}}
\Mdef{\hf}{\hat{f}}
\Mdef{\hA}{\hat{A}}
\Mdef{\hH}{\hat{H}}
\Mdef{\hJ}{\hat{J}}
\Mdef{\hM}{\hat{M}}
\Mdef{\hP}{\hat{P}}
\Mdef{\hQ}{\hat{Q}}

\Mdef{\thetab}{\overline{\theta}}
\Mdef{\phib}{\overline{\phi}}

\Mdef{\uA}{\underline{A}}
\Mdef{\uB}{\underline{B}}
\Mdef{\uC}{\underline{C}}
\Mdef{\uD}{\underline{D}}

\Mdef{\bolda}{\mathbf{a}}
\Mdef{\boldb}{\mathbf{b}}
\Mdef{\bfD}{\mathbf{D}}

\Mdef{\fm}{\frak{m}}
\Mdef{\fp}{\frak{p}}

\Mdef{\eps}{\epsilon}

\newcommand{\cell}{\mathrm{Cell}}

\newcommand{\BPn}{BP\langle n \rangle}
\newcommand{\ext}{\mathrm{Ext}}

\def\tz{tikzpicture}

\def\Vb#1{{\overline{V_{#1}}}}

\def\sb{{$\ssize\bullet$}}
\newcommand{\tmf}{\mathrm{tmf}}
\newcommand{\Ftwo}{\mathbb{F}_2}
\input{xypic}
\newtheorem{fig}[equation]{Figure}
\begin{document}
\title{Gorenstein duality and Universal coefficient theorems}
\author{Donald M. Davis}
\address{Department of Mathematics, Lehigh University\\Bethlehem, PA 18015, USA}
\email{dmd1@lehigh.edu}

\date{May 23, 2022}

\keywords{Universal Coefficient Theorem, Gorenstein ring, duality,
  Brown-Peterson(co)homology}
\thanks {2020 {\it Mathematics Subject Classification}: 55U20, 55U30,
  55N20, 55P43, 18G15, 13H10.}
\author{J.P.C.Greenlees}
\address{Mathematics Institute, Zeeman Building, Coventry CV4, 7AL, UK}
\email{john.greenlees@warwick.ac.uk}

\date{}

\begin{abstract}
We describe a duality phenomenon for cohomology 
theories with the character of Gorenstein rings.
\end{abstract}

\maketitle

\tableofcontents

\section{Introduction}
We describe a Universal Coefficient Theorem relating homology and cohomology of suitable 
torsion spaces when the coefficient ring $R_*$ has good homological
properties, and we go on to lift this to a highly structured
equivalence when the cohomology theory is sufficiently nice. 
For example, given a cohomology theory $R^*(\cdot)$ 
whose coefficient ring $R_*$ is Gorenstein of shift $a$ (see Section
\ref{sec:Gor}) with  $R_0=\Z_{(p)}$ the statement is as follows.
For spaces $X$ with $R_*(X)$ a torsion module (i.e., so that each
element is annihilated by some power of the maximal ideal), there is an isomorphism
$$R^*(X)\cong \Sigma^a (R_*(X))^\vee$$
of $R_*$-modules, where $M^\vee=\Hom(M,\Z/p^{\infty})$ is the Pontryagin dual. 
We will explain that this applies in particular when $R_*$ is a 
polynomial ring on finitely many generators, and how to find the shift 
$a$ in that case. In particular it applies to the well-known chromatic
Johnson-Wilson theory  $\BPn$,  to give the striking duality phenomenon proved by the first
author \cite{DavisBPn}, which was motivated by his work with W.S.Wilson regarding the 2-local $ku$-homology and $ku$-cohomology groups of the Eilenberg-MacLane space $K(\Z/2,2)$. 

The proof  is based on the homological behaviour of the
coefficient ring $\BPn_*$. When \cite{DavisBPn} was posted on the arXiv, the second author
recognized the statement as following from a form of Gorenstein duality
and wrote down the proof of a statement in a structured context. The two
proofs are based on the same piece of algebra, but in rather different
contexts so that their overlap in applicability is actually rather
small.

The authors decided that combining the two papers would be to the advantage
of both, by giving generality and perspective as well as specific
examples. The present paper describes the algebra behind both results, and then
develops it in the two contexts: the first gives conclusions in terms of
homology and cohomology groups and the second, when it applies,
enhances this to a conclusion in the derived category.

The first author is grateful to Andy Baker, Andrey Lazarev, Doug Ravenel, Chuck Weibel, and Steve Wilson for helpful suggestions. 

\section{Gorenstein rings}
\label{sec:Gor}

In this section we remind the reader of some well known facts about
a  graded commutative local Noetherian ring $A$ with residue field 
$k$ (see \cite{BrodmannSharp} for general background). We will soon apply them in our topological
context.

\begin{defn}
  \label{defn:Gora}
  We say that $A$   is {\em Gorenstein} if
   any one of the following three equivalent conditions hold
   \begin{enumerate}
    \item  $A$ is of finite injective
dimension as a module over itself
\item The Ext groups $\Ext_A^i(k,A)$ are non-zero for finitely many $i$ or
\item The Ext groups are 
  $$\Ext_A^i(k,A)=
  \dichotomy{\Sigma^{b}k& \mbox{ for } i=n}
  {0&\mbox{ otherwise}}$$
  for some $b$, where $n$ is the Krull dimension of $A$.
\end{enumerate}
The condition that we will make direct use of is  the third.

If  $A$ is Gorenstein then $a=b-n$ is called the {\em Gorenstein shift}
of $A$. 
\end{defn}

\begin{remark}
(i)  For some purposes the bigrading of $\Ext^*_A(k, A)$ is significant,
  but for us only the total degree will be relevant. 

 (ii) We include the case of an ungraded ring as a graded ring entirely in
  degree 0. In this case,  each Ext group is ungraded and the condition
  in 
  Part 3 necessarily has $b=0$.
  \end{remark}

\begin{example}
  \label{eg:Gor}
(i) If $A=k[x_1, \ldots , x_n]$ with all the $x_i$ of positive degree then $A$ is Gorenstein with 
$b=-\sum_i|x_i|$ and $a=b-n$. 

(ii) If $A=K [x_1, \ldots , x_n]$ for an ungraded Gorenstein local 
ring $K$ of shift $c$ with all the $x_i$ of positive degree  then $A$ is Gorenstein of shift $b+c-n$ where $b=-\sum_i|x_i|$. 
\end{example}

In this context, the Gorenstein condition automatically implies
a  duality statement. For this we make use of the injective hull $I(k)$ of
the residue field, and the Matlis dualization process for $A$-modules $M$ defined by 
$$M^{\vee}=\Hom_A(M, I(k)). $$

\begin{example}
  (i) If $A=k[x_1, \ldots , x_n]$ as above then $I(k)=\Hom_k(A,k)$ and 
  the Matlis dual has a simple description 
  $$M^\vee=\Hom_A(M, I(k))=\Hom_k(M, k)$$

  (ii) If $A=K[x_1, \ldots , x_n]$ as above then we can express the
  duality for $A$ in terms of that for $K$. Indeed, $I_A(k)=\Hom_K(A,I_K(k))$ and 
  the Matlis dual has a simple description 
  $$M^\vee=\Hom_A(M, I_A(k))=\Hom_K(M, I_K(k))$$
 \end{example}

With Matlis duality in hand, we can state the Gorenstein duality
enjoyed by torsion modules over a Gorenstein ring. 

\begin{lemma}
  \label{lem:Gor}
  If $A$ is a Gorenstein graded ring and $M$ is a torsion  $A$-module
  then
  $$\Ext^i_A(M, A)=\dichotomy
  {\Sigma^{b}M^\vee & \mbox{ if } i=n}
  {0  & \mbox{otherwise}}$$
\end{lemma}

\begin{proof}
This is a standard consequence of the Gorenstein condition. To give a
proof, one may consider the stable Koszul complex
$$K_{\infty}^\bullet (\fm)=(A\lra A[\frac{1}{y_1}])\tensor_A
\ldots \tensor_A (A\lra A[\frac{1}{y_s}])$$
associated to the maximal ideal  $\fm=(y_1, \ldots, y_s). $
The map $K_{\infty}^{\bullet}(\fm)\lra A$ induces a weak equivalence
$\Hom_A(M, \cdot)$ in the derived category. Now we use
$H^*_{\fm} (A)=H^n_\fm (A) =\Sigma^b I(k)$ by 
\cite[13.3.4]{BrodmannSharp}. It follows that 
$$\Ext^*_A(M, A)\simeq \Ext^n_A(M,  \Sigma^b I(k)) $$
as required.
  \end{proof}

\section{Consequences for cohomology theories}
\label{sec:coeffs}
Suppose now that $R$ is a ring spectrum (i.e., an $A_{\infty}$-ring)
so that $R_*$ is a Gorenstein commutative local Noetherian  ring of shift $a$ in the
sense of Section \ref{sec:Gor} and $R_0=K$. We immediately
have a universal coefficient theorem for spectra $X$ with $R_*(X)$
torsion.

\begin{thm}
  \label{thm:coeffs}
  If $R$ is a ring spectrum with $R_*$ Gorenstein of shift $a$ then
  for any $X$ with $R_*(X)$ torsion,   there is an isomorphism
  $$R^*(X)\cong \Sigma^a R_*(X)^{\vee}. $$
\end{thm}

\begin{proof}  
  By \cite[Corollary, p.257]{Rob} or \cite[IV.4.1]{EKMM},  if $R$ is an $A_\infty$ ring spectrum, there is a Universal Coefficient Spectral Sequence 
  $$\ext_{R_*}^{s,t}(R_*X,R_*)\Rightarrow R^{s+t}X.$$

 From Lemma \ref{lem:Gor},  the spectral sequence must
collapse, as it is confined to the single column $s=n$, and the
Lemma \ref{lem:Gor} gives the values. 
\end{proof}

\begin{example}
The  $p$-local Johnson-Wilson spectrum \cite{JW} $R=\BPn$ has
coefficient ring 
$R_*=\pi_*(R)=\Z_{(p)}[v_1,\ldots,v_n]$, with $|v_i|=2(p^i-1)$.

According to Example \ref{eg:Gor} (ii) this is a Gorenstein local
ring. Here $b=-D$ where $D=\sum_i |v_i|=2((p^{n+1}-1)/(p-1)-(n+1))$ is
the sum of the degrees of the generators and $c=-1$ is the Gorenstein
shift of $K= \Z_{(p)}\lra \Fp$, so that the Gorenstein shift of
$R_*\lra \Fp$ is  $a=b+c-n=-D-n-1$, and Matlis duality is
$(\cdot)^{\vee}=\Hom_{\Z_{(p)}} (\cdot , \Z/p^{\infty})$.
We may apply the general result since, by \cite[Corollary 3.5]{BJ}, 
$\BPn$ can be realised as an $A_{\infty}$-ring spectrum, so  we may
apply
Theorem \ref{thm:coeffs}.

Thus, if $R=\BPn$ and $R_*(X)$ is torsion, 
 there is an isomorphism of right $R_*$-modules 
$$R^*(X)\cong (R_*(\Sigma^{D+n+1}X))^\vee.$$

  \end{example}

  It may be worth making explicit the two simplest cases
  \begin{example}
 If $n=0$ then $BP\langle 
0\rangle=H\Z_{(p)}$ and $D=n=0$,  we have 
$$H^*(X; \Z_{(p)})=\Sigma^{-1}\Hom(H_*(X; \Z_{(p)}), 
\Z/p^\infty),$$
whenever $H_*(X; \Z_{(p)})$ is $p$-torsion. 
\end{example}

  \begin{example}
 If $n=1, p=2$ then $BP\langle 
1\rangle=ku$ is 2-local connective $K$-theory and $D=2, n=1$. We  thus
have 
$$ku^*(X)=\Sigma^{-4}\Hom(ku_*(X), 
\Z/2^\infty),$$
whenever $ku_*(X)$ is $(2, v_1)$-torsion. 
\end{example}

We illustrate this example with $X=K(\Z/2,2)$ in Section \ref{explsec}
to see that, even in this case, the duality statement is of considerable interest. 

\begin{remark}
  We have focused on connective theories, but this is not
  necessary. For instance, we may take Examples  \ref{eg:Gor} and adjoin a Laurent
  variable $u$ of positive degree $t$ to make the ring periodic and replace the
  field $k$ by the graded field $k[u,u^{-1}]$. 

  Localizing the previous argument we see  $A=k[x_1, \ldots , x_n]
  [u,u^{-1}]$  remains Gorenstein of shift 
$b=-\sum_i|x_i|$ and $a=b-n$, but now the shift is only defined 
mod $t$.  Similarly  $A=K [x_1, \ldots , x_n][u,u^{-1}] $
  is Gorenstein of shift $b+c-n$ where $b=-\sum_i|x_i|$, and again
  the shift is only defined mod $t$. 

This is relevant to certain well-known chromatic homotopy theories. 
The Johnson-Wilson theory $E(n+1)$ has
coefficient ring $E(n+1)_*=\Z_{(p)}[v_1, v_2, \ldots , v_{n}][v_{n+1},
v_{n+1}^{-1}]$. This is Gorenstein with the same shift as $BP\lr{n}_*=\Z_{(p)}[v_1, v_2, \ldots , v_{n}]$, but now only defined modulo $2(p^{n+1}-1)$.

Another example of this kind is the Lubin-Tate theory $E_n$ with coefficients
$(E_{n+1})_*=W(\mathbb{F}_{p^n})[[u_1, \ldots , u_n]][u, u^{-1}]$, with
$u_i$ of degree 0 and $u$ of degree 2 is Gorenstein of shift $-n-1$.
 \end{remark}

\section{An example for connective $K$-theory, with $\protect X=K(\Z /2,2)$.}\label{explsec}
In \cite{W} and \cite{DW},  Wilson and the first author
gave  partial calculations of $ku_*(K_2)$, where $K_2=K(\Z/2,2)$,
in their studies of Stiefel-Whitney classes. In \cite{DW2}, these
authors made a complete calculation of $ku^*(K_2)$. Using Theorem
\ref{thm:coeffs}, we can now give a complete determination of
$ku_*(K_2)$:
it is torsion since it contains no infinite groups or infinite
$v_1$-towers \cite{DW}. 

The work in \cite{DW} and \cite{DW2} was done using the Adams spectral sequence. It is interesting to compare the forms of the two Adams spectral sequence $E_\infty$ calculations. What appears as an $h_0$ multiplication in one usually appears as an exotic extension (multiplication by 2 not seen in Ext) in the other. We illustrate here with corresponding small portions of each. The portion of $ku^*(K_2)$ in Figure \ref{fig1} is called $A_3$ in \cite{DW2}. Note that in our $ku^*$ chart, indices increase from right to left. Exotic extensions  appear in red. One should think of the dual of the $ku_*$ chart as an upside-down version of the chart. The dual of the element in position $(30,7)$ in Figure \ref{fig2} is in position $(34,0)$ in Figure \ref{fig1}.

\bigskip 
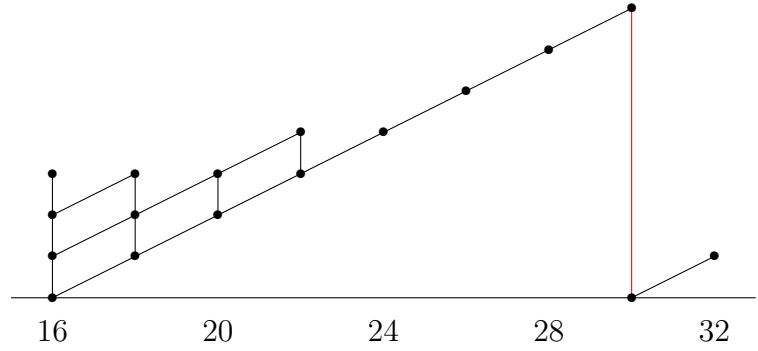
\begin{minipage}{6in}
\begin{fig}\label{fig1}

{\bf A portion of $ku^*(K_2)$}

\begin{center}
\begin{\tz}[scale=.55]
\draw (0,0) -- (2,1) -- (2,0) -- (16,7); 
\draw (-1,0) -- (17,0); 
\draw (10,0) -- (16,3); 
\draw (14,0) -- (16,1) -- (16,0); 
\draw [color=red] (10,0) -- (10,4); 
\draw [color=red] (14,0) -- (14,2); 
\draw [color=red] (16,1) -- (16,3); 
\draw [color=red] (12,1) -- (12,5); 
\draw [color=red] (14,2) -- (14,6); 
\draw [color=red] (16,3) -- (16,7); 
\node at (0,-.8) {$36$}; 
\node at (4,-.8) {$32$}; 
\node at (8,-.8) {$28$}; 
\node at (12,-.8) {$24$}; 
\node at (16,-.8) {$20$}; 
\node at (0,0) {\sb}; 
\node at (2,0) {\sb}; 
\node at (2,1) {\sb}; 
\node at (4,1) {\sb}; 
\node at (6,2) {\sb}; 
\node at (8,3) {\sb}; 
\node at (10,4) {\sb}; 
\node at (12,5) {\sb}; 
\node at (14,6) {\sb}; 
\node at (16,7) {\sb}; 
\node at (10,0) {\sb}; 
\node at (12,1) {\sb}; 
\node at (14,2) {\sb}; 
\node at (16,3) {\sb}; 
\node at (14,0) {\sb}; 
\node at (16,0) {\sb}; 
\node at (16,1) {\sb}; 
\end{\tz}
\end{center}
\end{fig}
\end{minipage}

\bigskip 

\bigskip 
\begin{minipage}{6in}
\begin{fig}\label{fig2}

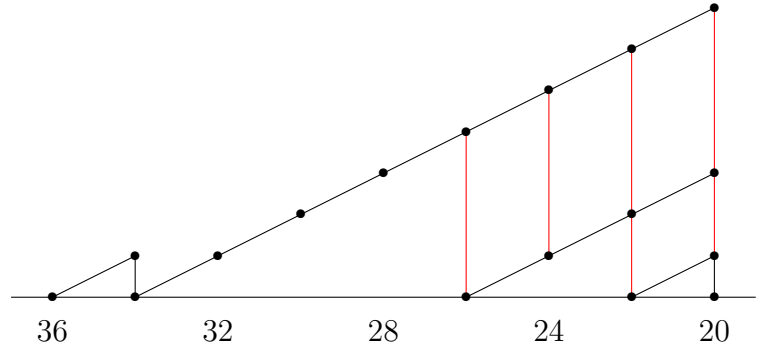
{\bf Corresponding portion of $ku_*(K_2)$}

\begin{center}

\begin{\tz}[scale=.55]
\draw (0,3) -- (0,0) -- (14,7); 
\draw (0,2) -- (2,3) -- (2,1); 
\draw (0,1) -- (6,4) -- (6,3); 
\draw (4,2) -- (4,3); 
\draw (14,0) -- (16,1); 
\draw [color=red] (14,0) -- (14,7); 
\node at (0,-.8) {$16$}; 
\node at (4,-.8) {$20$}; 
\node at (8,-.8) {$24$}; 
\node at (12,-.8) {$28$}; 
\node at (16,-.8) {$32$}; 
\draw (-1,0) -- (17,0); 
\node at (0,0) {\sb}; 
\node at (0,1) {\sb}; 
\node at (0,2) {\sb}; 
\node at (0,3) {\sb}; 
\node at (2,1) {\sb}; 
\node at (2,2) {\sb}; 
\node at (2,3) {\sb}; 
\node at (4,2) {\sb}; 
\node at (4,3) {\sb}; 
\node at (6,3) {\sb}; 
\node at (6,4) {\sb}; 
\node at (8,4) {\sb}; 
\node at (10,5) {\sb}; 
\node at (12,6) {\sb}; 
\node at (14,7) {\sb}; 
\node at (14,0) {\sb}; 
\node at (16,1) {\sb}; 
\end{\tz}
\end{center}
\end{fig}
\end{minipage}
\def\line{\rule{.6in}{.6pt}}

\section{Gorenstein ring spectra and Gorenstein
  duality}  

We now suppose that the representing spectrum $R$ of our cohomology
theory admits the structure of a commutative ring ($E_{\infty}$-ring). In this case we now
describe how the duality statement may be lifted to one at the level
of $R$-modules. When this applies, the duality statement in Theorem
\ref{thm:coeffs} can be deduced by passage to homotopy. There
are many examples where $R$ is a ring with Gorenstein coefficients
$R_*$ but where $R$ is not known to be realized as a commutative ring
($R=\BPn$ with $n\geq 3$ for instance), so that
the  approach of Section \ref{sec:coeffs} gives the best available results. On the other
hand there are many examples where $R$ is 
Gorenstein in a suitable derived sense without the coefficient ring
$R_*$ being Gorenstein,
and in this case we obtain results without counterparts in the setting
of Section \ref{sec:coeffs}. The results in this section build on
\cite{AdualGdual} in the context of \cite{hogor}.

\subsection{The Gorenstein condition}
We suppose given a commutative ring spectrum $R$ and a map $R\lra
\ell$. In our examples $\ell$ will either be the local ring $K=R_0$ or
its residue field $k$. Translating the third condition of Definition \ref{defn:Gora} into
the derived category, we obtain the definition of a Gorenstein ring spectrum.

\begin{defn} \label{defn:Gor}
  \cite[8.1]{hogor}
  We say that $R\lra \ell$ is {\em Gorenstein of shift $a$} if there is an equivalence
  $$\Hom_R(\ell,R)\simeq \Sigma^a\ell$$
of left $R$-modules. 
\end{defn}

\begin{remark}
The paper \cite{hogor} also requires a finiteness condition before $R\lra
\ell$ can be called Gorenstein. We will impose a slightly stronger
finiteness condition later.
\end{remark}

\begin{notation}
  In all the examples we consider here, $\ell$ will be an Eilenberg-MacLane
spectrum. We  will  use the same letter $\ell$ to denote both the
classical ring
and the Eilenberg-MacLane spectrum,   relying on context to determine
which is intended at any point.

\end{notation}

\begin{example}
We note that if $R_*$ is itself Gorenstein (as in Examples
\ref{eg:Gor}) it follows that $R\lra k$ is Gorenstein with the same
shift. In particular, this applies to the examples of polynomial rings
over $k$ or over a Gorenstein local ring $K$.
  


  \end{example}

  \subsection{Anderson duality and Brown-Comenetz duality}
Suppose  $R$ is connective and $K=R_0$. For any injective $K$-module
$I$ and any $R$-module $Y$ we may define the Brown-Comenetz dual spectrum $I^Y$ to be the
$R$-module defined by the formula
  $$[M, I^Y]^R_*=\Hom_K(\pi_*(M\tensor_R Y), I), $$
for $R$-modules $M$; this defines $I^Y$ as the representing object
since the  right hand side is a cohomology theory ($I$ is injective
over $K$). Note the special case $M=R\sm X$,  when 
  $$[X, I^Y]=\Hom_K(\pi_*(X\sm Y), I), $$

Now if   $K$ is of injective dimension $ 1 $ over itself, then
  we can choose an injective resolution of $K$-modules
  $$0\lra K \lra I\lra J\lra 0.  $$
  (for example if $K=\Z_{(p)}$ then we may take $I=\Q$ and $J=\Q/\Z_{(p)}=\Z/p^\infty$). 
 For an $R$-module $Y$, there are Brown-Comenetz duals $I^Y$ and $J^Y$ with respect to $I$ and $J$, and we define the Anderson dual with respect to $K$ via the cofibre sequence
 $$K^Y\lra I^Y\lra J^Y.$$
 Up to equivalence $K^Y$ is independent of the resolution. One can
 immediately find a short exact sequence for maps into the Anderson dual.

  \begin{cor}
    There is a short exact sequence
$$  0\lra \Ext_K(\Sigma Y^R_*(M),K)\lra  [M,K^Y]\lra \Hom_K(Y^R_*(M),K)\lra 0$$
In particular, we have an isomorphism
$$[K, K^R]_*\cong \Hom_K(K,K)=K $$
\end{cor}

\subsection{Two finiteness conditions}
There are numerous cases of interest where, $R\lra \ell$ satisfies the strong condition that $\ell$ is 
small over $R$ in the sense that $\ell$ is finitely built from 
$R$. If $R$ is a conventional commutative local Noetherian
ring this is equivalent to  $R$ being regular.

The relevant results continue to hold under  a much weaker condition. We require
$\ell$ to be {\em   proxy small} \cite[4.6]{hogor} in the sense that there
is a small object
$\widehat{\ell}$ so 
that $\ell$ finitely builds $\widehat{\ell }$ and $\widehat{\ell}$ builds $\ell$. If $R$ is
a conventional commutative local Noetherian ring this always holds,
and we may take $\widehat{\ell}$ to be the Koszul complex associated to any
finite generating set for the maximal ideal.

\subsection{Cellularization}
\label{subsec:cell}
Let $\cE=\Hom_R(\ell ,\ell)$ and note that for any $R$-module $M$,
$\Hom_R(\ell ,M)$ is a right $\cE$-module. We have an evaluation map
$$\epsilon: \Hom_R(\ell, M)\tensor_{\cE}\ell \lra M.  $$
We note that $\Hom_R(\ell,M)$ is built from $\cE $ and hence
$\Hom_R(\ell ,M)\tensor_{\cE}\ell$ is built from $\ell$. The evaluation map is
thus a map from a $\ell$-cellular object to $M$,  and 
it is a $\ell$-equivalence provided $\ell$ is proxy small (\cite[6.10
and  6.14]{hogor}). Thus 
$$\cell_\ell M:=\Hom_R(\ell,M)\tensor_{\cE}\ell $$
is $\ell$-cellularization. 

\begin{example}
Suppose $R$ is connective with $R_0=\Z_{(p)}$ and $R_n$ is a finitely
generated $\Z_{(p)}$-module in each degree.

(i) If $K=\Z_{(p)}$ and $R\lra K$ is an isomorphism in $\pi_0$,
then we see that $K^R$ is built from $K^{K} \simeq K$, 
and so $\cell_{K }K^R\simeq K^R$.

(ii) Now suppose $k=\Fp$ and that in $\pi_0$ the map  $R\lra k=\Fp$ is
projection $\Z_{(p)}\lra \Fp$ onto the residue field.

Since $\Q^R$ is rational, we find $\cell_{\Fp}\Q^R\simeq *$. This gives
$$\cell_{\Fp}\Z_{(p)}^R\simeq \Sigma^{-1} \cell_{\Fp}(\Z/p^\infty)^R\simeq
\Sigma^{-1}(\Z/p^{\infty})^R,$$
the desuspension of the  Brown-Comenetz dual of $R$.
\end{example}

\subsection{Algebraic criteria for cellularity}
There are a number of cases where we can give criteria for cellularity
by looking at coefficients. Suppose  $\fm=\ker (R_*\lra k)=(y_1,
\ldots, y_s) $ and write
$$\Gamma_{\fm}R:=\fib (R\lra R[\frac{1}{y_1}])\tensor_R
\cdots \tensor_R \fib (R\lra R[\frac{1}{y_s}]), $$
for the stable Koszul complex and $\Gamma_\fm M=\Gamma_{\fm}R\tensor_RM$. By construction there is a spectral
sequence
$$H^*_{\fm}(R_*; M_*)\Rightarrow \pi_*(\Gamma_{\fm}M),  $$
so that in particular if  $M=R$ and $R_*$ is Cohen-Macaulay,
this collapses to an isomorphism
$$H^n_{\fm}(R_*)=\Sigma^n \pi_*(\Gamma_{\fm}R). $$
We say that $R$ has {\em algebraic $k$-cellularization} if the stable
Koszul complex gives the cellularization: 
$$\Gamma_{\fm}R\simeq \cell_k R. $$
We note that in this case $k$ is proxy small with the unstable Koszul
complex $R/y_1\tensor_R \cdots \tensor_R R/y_s$ as witness. 

\begin{lemma}
If $R$ has algebraic $k$-cellularization then $M$ is $k$-cellular if
and only if $M_*$ is a torsion $R_*$-module. 
  \end{lemma}

  \begin{proof}
    If $M$ is $k$-cellular then $M\simeq \cell_kM\simeq \Gamma_\fm R
    \tensor M$, and the spectral sequence for calculating $M_*$ is
    finite with a torsion $E_2$-term so $M_*$ is torsion.

    Conversely if $M_*$ is torsion then
    $H^*_\fm(M_*)=H^0_{\fm}(M_*)=M_*$ and the map $\Gamma_\fm M \lra
    M$ is an isomorphism in homotopy and hence a weak equivalence. 
  \end{proof}

  \begin{remark}
If $k$ is a finite field then an $R_*$-module $M_*$ is torsion if and
only if it is {\em locally finite} in the sense that the submodule
generated by an  element $x\in M_*$ is a finite set. 
    \end{remark}

\begin{example}
(i)  Suppose $R_0=\Z_{(p)}$ with $R_*=\Z_{(p)}[x_1, \ldots , x_n]$  (for
  example $R=BP\langle n \rangle$), so that $\fm=(p, x_1, \ldots ,
  x_n)$. In this case $R\lra k$  is regular and has algebraic $k$-cellularization.

  If is clear that $\Gamma_\fm R $ is the
  $\widehat{k}$-cellularization where
  $$\widehat{k}=R/p \tensor_R R/x_1\tensor_R \cdots \tensor_R R/x_n. $$
 In this case $\widehat{k}\simeq k$ so that $R$ has algebraic
 $k$-cellularization. 


(ii)  It is shown in \cite[5.1]{AdualGdual} that this extends to the
case that $R\lra k$ is proxy regular and has algebraic cellularization
if $R_*$ is a hypersurface ring.

(iii)  It is proved in \cite[5.2]{AdualGdual} or \cite{tmfGor} that
the ring spectrum $\tmf$ of 2-local topological modular forms
has algebraic $k$-cellularization. 

(iv)  The proof of \cite[9.3]{hogor} shows that for any compact Lie
group $G$ the spectrum $R=C^*(BG) =\map (BG, k)$ of cochains on $BG$
has algebraic $k$-cellularization.
\end{example}


\subsection{Gorenstein duality}
\newcommand{\uk}{\underline{k}}
We continue to suppose that $R$ is a connective commutative ring
spectrum with $R_0=K$ a local ring with residue field $k$. 
We say that $R\lra k$ has {\em Gorenstein
duality of shift $a$} if we have an equivalence
$$\Gamma_k R\simeq \Sigma^a I(k)^R, $$
and we say $R\lra K$ has {\em Gorenstein duality of shift $a$} if 
$$\Gamma_K R\simeq \Sigma^aK^R. $$
In this section we recall that under favourable circumstances, if
$R\lra k$ is Gorenstein of shift $a$ then it has Gorenstein duality of
shift $a$, and similarly for $R\lra K$.

If $R\lra k$ is Gorenstein of shift $a$ in the sense of
\ref{defn:Gor}, we have
$$\Hom_R(k,R)\simeq \Sigma^a k\simeq \Hom_R(k, \Sigma^a I(k)^R). $$
Provided the composite $\Hom_R(k, R)\simeq \Hom_R(k, \Sigma^ak^R)$ is
an isomorphism of right $\cE$-modules (for example if $\cE$ has a
unique action on $k$),  then we may apply $\tensor_{\cE} k$, and
provided $k$ is proxy small,
Subsection \ref{subsec:cell} shows that we have an equivalence
$$\cell_kR\simeq \Sigma^a\cell_k I(k)^R $$
of left $R$-modules.

Similarly 
$$\cell_KR\simeq \Sigma^a\cell_K K^R .$$

\begin{example}
  If $R_0=K=\Z_{(p)}$, $k=\Fp$ and $R_*=K [x_1, \ldots, x_n]$ (for example 
  $R=BP\langle n \rangle$) then $R_*\lra K$
  is Gorenstein of shift $-D-n$, so is $R\lra K$. The ring
  $\cE$ is exterior on generators of 
degree $-|x_1|-1, \ldots , -|x_n|-1$. This has a unique action on 
$K$ by the argument of \cite[3.9]{hogor}. We deduce
$$\Gamma_K R\simeq \Sigma^{-D-n}\Gamma_{K}K^R\simeq 
\Sigma^{-D-n}K^R. $$

We now apply $\Gamma_{p}$ to both sides, noting
$\Gamma_p\Gamma_\Z=\Gamma_{\Fp}$ and see
$$\Gamma_{\Fp}R\simeq \Sigma^{-D-n}\Gamma_{p}K^R\simeq 
\Sigma^{-D-n-1}(\Z/p^{\infty})^R. $$
This is the statement that $R \lra \Fp$ has Gorenstein duality of
shift $-D-n-1$. 
  \end{example}

  \begin{remark}
It would be preferable to take $k=\Fp$ and argue directly. The map
$R\lra \Fp$ is  Gorenstein of shift $-D-n-1$ and we would like to
deduce Gorenstein duality of the same shift. However this would require a unique action of
$\Hom_R(\Fp, \Fp)$ on $\Fp$, and the degrees of generators do not make
this obvious. 
    \end{remark}

    \subsection{The Universal Coefficient Theorem}
    Finally we may deduce the duality statement for a commutative ring
    spectrum with Gorenstein duality.

\begin{thm}
    Suppose the $R\lra k$ has Gorenstein duality of shift $a$. It then
    follows that if $M$ is $k$-cellular we     have an equivalence
    $$\Hom_R(M, R)\simeq \Sigma^a I(k)^M. $$
 \end{thm}

  \begin{proof}
 Since $M$ is $k$-cellular we have 
    $$\Hom_R(M,R) \simeq \Hom_R(M, \cell_kR) \simeq \Hom_R(M, \Sigma^a
    I(k)^R) \simeq \Sigma^a I(k)^M$$
\end{proof}

The special case $M=R\sm X$ is of particular interest,
especially when cellularity can be detected in terms of
coefficients. 

\begin{cor}
 Taking $M=R\sm X$, we see that if $R_*X$ is a torsion module we have an isomorphism 
    $$R^*(X)=\Sigma^a\Hom_K(R_*X, I(k)). $$
    \end{cor}

    \begin{proof} We may calculate
\begin{multline*}
  R^*(X)=[X,R]_*=\\
  [R\sm X, R]^R_*\simeq [R\sm X, \Gamma_{\Fp} R]^R_*
\simeq [R\sm X, \Sigma^a(\Z/p^\infty)^R]^R_*
\simeq [ X, \Sigma^a(\Z/p^\infty)^R]_*\\
=\Sigma^a \Hom_\Z(R_*(X), \Z/p^\infty) 
\end{multline*}
\end{proof}

\begin{example}
Suppose $R_0=K=\Z_{(p)}$, $k=\Fp$ and $R_*=K [x_1, 
\ldots, x_n]$ (for example   $R=BP\langle n \rangle$) with $a=-D-n-1$
we recover Theorem \ref{thm:coeffs}. We note that it is only known that $\BPn$
admits the structure of a commutative ring for $n\leq 2$. 
\end{example}

\begin{example}
The spectrum  $R=\tmf$ of 2-local topological modular forms has
$R_0=K=\Z_{(2)}$ and 
$k=\Ftwo$. This has Gorenstein duality in the form
$$\Gamma_{\Ftwo}\tmf \simeq \Sigma^{-23}(\Z/2^{\infty})^\tmf.$$
A result of this type was first proved by Mahowald-Rezk \cite{MR}; a
proof of precisely the statement here, along with a discussion of
alternative approaches can be found in \cite[4.8]{tmfGor}. From this
we deduce 
$$\tmf^*(X) =\Sigma^{-23}\Hom_{\Z_{(2)}}(\tmf_*(X), \Z/2^{\infty})$$
whenever $\tmf_*(X) $ is $\tmf_*$-torsion. 
\end{example}

\begin{example}
If $G$ is a finite group, we may take $R=C^*(BG; k)$ with $R_0=k$. By \cite[10.3]{hogor} this has
Gorenstein duality with shift 0. In the present case the Matlis dual
is the vector space dual, so that Gorenstein duality takes the form
$$\Gamma_{k}C^*(BG) \simeq C_*(BG).$$
A plentiful supply of modules comes from $G$-spaces $X$, where we may
form  $M=C_*(EG\times_GX)$; this is obviously torsion since it is zero
in negative degrees. Inserting this into the theorem we see that for
any  $G$-space there is a spectral sequence
$$\Ext_{H^*(BG)}^{*,*}(H_*(EG\times_G X), H^*(BG))\Rightarrow
H^*_G(EG\times_G X), $$
which is in the form of a Universal Coefficient Theorem relating
homology and cohomology of the Borel construction. One could approach
this in an equivariant context, where the form would seem very
familiar. 
\end{example}

\subsection{Variations}
We continue to assume $R\lra k$ has Gorenstein duality of shift $a$,
and we could also have considered the implications of Gorenstein duality for
$R\lra K$ (which is of shift $a+1$). This would take the form
$$R^*(X)\sim \Sigma^{a+1} \Hom_K(R_*(X), K)$$
where the symbol $\sim$ indicates that on the right $\Hom_K$ and $\Ext_K$
will be involved when $\widehat{R}_*(X)$ is not projective over
$K$. The torsion requirement on $R_*(X)$ would now refer not to the
maximal ideal but to the ideal $J=\ker(R_*\lra K)$.

More explicitly, if $R_*(X)$ is $J$-power torsion, there is a short exact sequence
$$0\lra \Sigma^a \Ext_K(R_*(X), K)\lra R^*(X)\lra \Sigma^{a+1} \Hom_K (R_*(X), K)\lra 0 . $$

We have focused on connective theories, and hence obtained universal
coefficient theorems when $R_*(X)$ is torsion. It is explained in
\cite{AdualGdual} that when $R$ is Gorenstein we may nullify $K$ to
form a new theory $\widehat{R}$, which can be thought of as splicing
together $R$ and $\cell_K R$, and which will usually not be
connective. For example if $R=ku$, with $a=-4$ then
$\widehat{R}=KU$, if $R=\tmf$ with $a=-23$ then  $\widehat{R}=\mathrm{Tmf}$, and
if $R=C^*(BG)$ then $\widehat{R}$ is the fixed points of the  Tate spectrum of $G$.

In favourable cases when $R\lra K$ has Gorenstein duality
with shift $a+1$ then we obtain an Anderson self-duality statement for $\widehat{R}$:
$$\widehat{R}\simeq \Sigma^{a+2} K^{\widehat{R}}.$$
This then gives a Universal Coefficient Theorem 
$$\widehat{R}^*(X)\sim \Sigma^{a+2} \Hom_K (\widehat{R}_*(K), K). $$
The symbol $\sim$ has the same meaning as above, but 
there is now no torsion requirement on $\widehat{R}_*(X)$.

\end{document}